\documentclass[11pt]{article}

\usepackage[margin=1in]{geometry}   
\usepackage{color} 
\usepackage{enumerate}
\usepackage{graphicx}              
\usepackage{amsmath}               
\usepackage{amsfonts}              
\usepackage{amsthm}                
\usepackage{amssymb} 
\usepackage[normalem]{ulem}
	
\usepackage{amssymb,amsfonts,amsthm,amsmath,graphicx,url}

  \usepackage{latexsym}
  \usepackage{graphicx, float, url} 
\newtheorem{thm}{Theorem}[section]
\newtheorem{lem}[thm]{Lemma}
  \newtheorem{const}[thm]{Construction}

\newtheorem{cor}[thm]{Corollary}

\newcommand{\mf}[1]{\mathbf{#1}}  
\newcommand{\ZZ}{\mathbb{Z}}      

\newcommand{\norm}[1]{\lVert#1\rVert}

\newcommand{\BL}[2]{{\left\lfloor\frac{#1}{#2}\right\rfloor }}
\def\smid{3n(n^2 - \epsilon(n))/(n-1)(\sqrt{2n}-7)^3}

\def\sr {\lfloor (s+r)/2\rfloor}

\def\ss {\lfloor s/2\rfloor}

\def\bsss {\left\lfloor \frac{s}{2}\right\rfloor}
\def\bssspl {\left\lfloor \frac{s+2}{2}\right\rfloor}
\def\rr {\lfloor r/2\rfloor}

\def\brrr {\left\lfloor \frac{r}{2}\right\rfloor}
\def\qq {\lfloor q/2\rfloor}
\def\cqq {\lceil q/2\rceil}

\def\nn {\lfloor n/2\rfloor}

\def\aw {{\overrightarrow{w}}}
\def\api {{\overrightarrow{\pi}}}
\def\aw {{\overrightarrow{w}}}

\def\b0{{\bf 0}}
\def\ba{{\bf a}}
\def\bb{{\bf b}}

\def\bu{{\bf u}}
\def\bv{{\bf v}}
\def\bw{{\bf w}}

\long\def\delete#1{}

\input amssym.def
\input amssym.tex

\newcommand{\be}{\begin{equation}}
\newcommand{\ee}{\end{equation}}
\newcommand{\bea}{\begin{eqnarray}}
\newcommand{\eea}{\end{eqnarray}}
\newcommand{\bean}{\begin{eqnarray*}}
\newcommand{\eean}{\end{eqnarray*}}

\begin{document}
	
	\title{Forwarding and optical indices of 4-regular circulant networks}

	\author{Heng-Soon Gan, Hamid Mokhtar\thanks{Corresponding author} and Sanming Zhou\thanks{{\scriptsize hsg@unimelb.edu.au (H.-S.~Gan), hmokhtar@student.unimelb.edu.au (H.~Mokhtar), smzhou@ms.unimelb.edu.au (S. Zhou)}} \\
\small School of Mathematics and Statistics, The University of Melbourne,\\ \small Parkville, VIC 3010, Australia}

	\maketitle
	\openup 0.5\jot 
	
	\begin{abstract} 
 An all-to-all routing in a graph $G$ is a set of oriented paths of $G$, with exactly one path for each ordered pair of vertices. The load of an edge under an all-to-all routing $R$ is the number of times it is used (in either direction) by paths of $R$, and the maximum load of an edge is denoted by $\pi(G,R)$. The edge-forwarding index $\pi(G)$ is the minimum of $\pi(G,R)$ over all possible all-to-all routings $R$, and the arc-forwarding index $\api(G)$ is defined similarly by taking direction into consideration, where an arc is an ordered pair of adjacent vertices. Denote by $w(G,R)$ the minimum number of colours required to colour the paths of $R$ such that any two paths having an edge in common receive distinct colours. The optical index $w(G)$ is defined to be the minimum of $w(G,R)$ over all possible $R$, and the directed optical index $\aw(G)$ is defined similarly by requiring that any two paths having an arc in common receive distinct colours. In this paper we obtain lower and upper bounds on these four invariants for $4$-regular circulant graphs with connection set $\{\pm 1,\pm s\}$, $1<s<n/2$. We give approximation algorithms with performance ratio a small constant for the corresponding forwarding index and routing and wavelength assignment problems for some families of $4$-regular circulant graphs.
	
	\end{abstract}

\section{Introduction} 

\subsection{Motivation and definitions}
	
Circulant graphs, or multi-loop networks as used in computer science literature, are basic structures for interconnection networks \cite{Bermond1995}. As such a lot of research on circulant graphs has been done in more than three decades, leading to a number of results on various aspects of circulant graphs \cite{Bermond1995, Fertin2004,Gauyacq1998,Gomez2007, Hwang2001, Mans2004, Stojmenovic1997,Thomson2008,Thomson2010,Thomson2012}. Nevertheless, our knowledge on how circulant networks behave with regard to information dissemination is very limited. For example, our understanding to some basic communication-related invariants for circulant graphs such as the arc-forwarding, edge-forwarding and optical indices is quite limited. The purpose of this paper is to study these invariants with a focus on circulant networks of degree 4. 

Given an integer $n \ge 3$, denote by $\ZZ_n$ the group of integers modulo $n$ with operation the usual addition. Given $S \subset \ZZ_n$ such that $ 0 \not \in S$ and $s\in S$ implies $-s\in S$, the \textit{circulant graph} $C_n(S)$ of order $n$ with respect to $S$ is defined to have vertex set $\ZZ_n$ such that $i, j \in \ZZ_n$ are adjacent if and only if $i-j\in S$. (In other words, a circulant graph is a Cayley graph on $\ZZ_n$.) In the case when $S = \{a, b, n-a, n-b\}$, where $a, b, n-a, n-b$ are pairwise distinct elements of $\ZZ_n$, $C_n(S)$ is a 4-regular graph (that is, every vertex has degree 4) and we use $C_n(a, b)$ in place of $C_n(S)$. In this paper we deal with circulant graphs $C_n(1, s)$ for some $s \in \ZZ_n \setminus \{-1, 0, 1, n/2\}$. (Note that when $n$ and $a$ are coprime, $C_n(a, b)$ is isomorphic to $C_n(1,s)$, where $s \equiv a^{-1}b\mod n$). Without loss of generality,  we assume $1<s<n/2$. Just like any other Cayley graph, $C_n(1, s)$ is vertex-transitive, that is, for any $i, j \in \ZZ_n$ there exists a permutation of $\ZZ_n$ that preserves the adjacency relation of $C_n(1, s)$ and maps $i$ to $j$. (In fact, for fixed $i, j$ this permutation can be chosen as $x \mapsto x+(j-i)$, $x \in \ZZ_n$ with operation undertaken in $\ZZ_n$.)

A network can be represented by an undirected graph $G=(V(G), E(G))$, where the node set $V(G)$ represents the set of processors or routers, and the edge set $E(G)$ represents the set of physical links. So we will use the words `graph' and `network' interchangeably. We assume the full duplex model, that is, an edge is regarded as two \textit{arcs} with opposite directions over which messages can be transmitted concurrently. A connection request (or a \textit{request} for short) is an ordered pair of distinct nodes $(x,y)$ for which a path $P_{x, y}$ with orientation from $x$ to $y$ in $G$ must be set up to transmit messages from $x$ to $y$. In this paper we only consider all-to-all communication, or equivalently, the \textit{all-to-all request set} for which one path from every node to every other node must be set up in order to fulfil communications. (In the literature other types of request sets have also been studied.) We call a set of paths $R = \{P_{x, y}: x, y \in V(G), x \not=y\}$ an \textit{all-to-all routing} (or a \textit{routing} for short) in $G$, where $P_{x, y}$ is not necessarily the same as $P_{y, x}$. The load of an edge $e$ of $G$ with respect to $R$, denoted by $\pi(G,R,e)$, is the number of paths in $R$ passing through $e$ in either directions. Similarly, the load of an arc $a$ of $G$ with respect to $R$, denoted by $\api(G,R,a)$, is the number of paths in $R$ passing through $a$ along its direction. Define 
\be
\label{eq:load}
\pi(G, R):=\max_{e \in E(G)} \pi(G,R,e),\quad \api(G, R):=\max_{a\in A(G)} \api(G, R, a),
\ee
where $A(G)$ is the set of arcs of $G$. Define
\be
\label{eq:pidef}
\pi(G):=\min_{R}\pi(G,R),\quad \api(G):=\min_{R}\api(G,R)
\ee
and call them the \emph{edge-forwarding} and \emph{arc-forwarding indices} of $G$ \cite{Beauquier1997,Heydemann1989}, respectively, where the minimum is taken over all routings $R$ for $G$. Obviously, we have  
	 \begin{equation}
	 \label{eq:pi}
	 \api(G)\geq\pi(G)/2.
	 \end{equation} 
The \textit{edge-forwarding index problem} is the one of finding $\pi(G)$ for a given graph $G$, and the \textit{arc-forwarding index problem} is understood similarly.

In practical terms, the edge-forwarding and arc-forwarding indices measure the minimum heaviest load on edges and arcs of a given network, respectively, with respect to all-to-all communication. If the network is all-optical, another important problem is to minimise the number of wavelengths used such that any two paths having an edge (or arc) in common are assigned distinct wavelengths. Regarding wavelengths as colours, these problems can be formulated as the following {\em path colouring problems}. Given a routing $R$ for $G$, an assignment of one colour to each path in $R$ is called an {\em edge-conflict-free colouring} of $R$ if any two paths having an edge in common (regardless of the orientation of the paths) receive distinct colours, and an {\em arc-conflict-free colouring} of $R$ if any two paths having an arc in common (with the same orientation as the paths) receive distinct colours. (An edge-conflict-free colouring is called \emph{valid} in \cite{Gargano2000}.) Define $w(G, R)$ ($\aw(G,R)$, respectively) to be the minimum number of colours required in an edge-conflict-free (arc-conflict-free, respectively) colouring of $R$. Define
\be
\label{eq:w}
w(G):=\min_{R} w(G,R),\quad \aw(G):=\min_{R}\aw(G,R)
\ee
and call them the {\em undirected} and {\em  directed optical indices} of $G$, respectively, 
where the minimum is taken over all routings $R$ for $G$.  
Since the number of colours needed is no less than the number of paths on a most loaded edge (or arc in the directed version), we have (see e.g.~\cite{Bermond2000})
	\begin{equation}
	\label{eq:piw}
	w(G)\ge\pi(G),\quad \aw(G)\ge\api(G).
	\end{equation}
In general, equality in (\ref{eq:piw}) is not necessarily true (see e.g.~\cite{Kosowski2009,Yuan2007}). The \textit{routing and wavelength assignment problem} is the problem of computing $w(G)$, and its \textit{oriented version} is the one of finding $\aw(G)$.

\subsection{Literature review} 

The study of the forwarding indices has been intensive in the literature. 
Heydemann et al. \cite{Heydemann1989} proposed the edge-forwarding index problem and obtained basic results on this invariant, including upper bounds for the Cartesian product of graphs. In \cite{Sole1994} it was proved that orbital regular graphs (which are essentially Frobenius graphs \cite{Fang1998} except cycles and stars) achieve the smallest possible edge-forwarding index. In \cite{Thomson2008,Thomson2010,Thomson2012}, Thomson and Zhou gave formulas for the edge-forwarding and arc-forwarding indices of two interesting families of Frobenius circulant graphs. The exact value of edge-forwarding index of some other graphs have also been computed, including Kn\"odel graphs \cite{Fertin2004} and recursive circulant graphs \cite{Gauyacq1998}. However, in general it is difficult to find the exact value or a good estimate of the edge-forwarding or arc-forwarding index of a graph, even  for some innocent-looking classes of graphs such as circulant graphs.  The authors of \cite{Xu2007} obtained lower and upper bounds on the edge-forwarding index of a general circulant graph. However, these bounds are difficult to compute in general.  Also, a uniform routing of shortest paths may not exist for circulant graphs, just as the case for Cayley graphs in general \cite{Heydemann1997}.
The reader is referred to the recent survey \cite{Xu2013} for the state-of-the-art on edge-forwarding and arc-forwarding indices of graphs.  

The routing and wavelength assignment problem has also received considerable attention due to the importance of optical networking.  
The authors of \cite{Beauquier1997} surveyed theoretical results and asked a few questions on the routing and wavelength assignment problem for a general request set as well as the all-to-all request set, especially for trees, rings, tori, meshes and hypercubes. 
In \cite{Gargano2000} a survey of results on several versions of the routing and wavelength assignment problem was given and the exact value of the directed optical index of stars was obtained.
The exact value of the directed optical index was computed for trees  \cite{Gargano1997}, rings \cite{Bermond2000}, trees of rings \cite{Beauquier:1999b}, hypercubes \cite{Bermond2000}, a few families of recursive circulant graphs \cite{Amar2001}, Cartesian sum of complete graphs \cite{Beauquier1999}, Cartesian product of paths with equal even lengths \cite{Beauquier1999}, and Cartesian product of some graphs such as rings \cite{Beauquier1999, Schroder1997}. While the exact value of the undirected optical index has been obtained for rings and hypercubes \cite{Bermond2000}, determining this invariant remains open for many families of graphs \cite{Beauquier1997}. For the undirected routing and wavelength assignment problem for arbitrary request sets on trees of rings, the best known result is a 2.75-approximation algorithm \cite{Bian2009} and a 2-approximation algorithm \cite{Deng2003} for a subfamily of such graphs.

In  \cite{Kosowski2009} it was proved that the problem of deciding whether $\pi(G)\le 3$ or $w(G) \le 3$ is NP-complete.    

\subsection{Main results}

{\em 
In what follows we assume that $n$ and $s$ are integers with $n \ge 5$ and $1 < s < n/2$, and $q$ and $r$ are integers defined by 
$$
q := \lfloor n/s\rfloor,\;\, n := q s + r.
$$ 
} 
Observe that $0 \le r < s$ and $q \ge 2$ as $s < n/2$. 
Denote by $x\oplus y$ ($x \ominus y$, respectively) the integer $x+y\mod n$ ($x - y\mod n$, respectively) between $0$ and $n-1$, where $x,y\in \ZZ_n$. Denote
$$
\epsilon(x) := \left\{ 
\begin{array}{ll}
1,\;\; \mbox{if $x$ is an odd integer}\\[0.2cm] 
0,\;\; \mbox{if $x$ is an even integer.}
\end{array}
\right.
$$

The first main result in this paper is as follows.
 
\begin{thm}
\label{thm:tr1}
The following hold:
\begin{itemize}
\item[\rm (a)]
if $2 \le  s \le \sqrt{n}  - 1$, then
\be
\label{eq:tf1}
\frac{n^2 - \epsilon(n)}{8(s+1)} \leq  \frac{\pi(C_n(1,s))}{2}  \le  \api(C_n(1,s)) \leq \frac{(n-r)( n+ r+2)+s^2}{8s};
\ee 
\item[\rm (b)]
if $ s = \sqrt{n}$, then
\be
\label{eq:tf2}
\frac{\sqrt{n}\ (n - 1)}{8}
 \leq 
 \frac{\pi(C_n(1,s))}{2}  \le  \api(C_n(1,s))
 \leq
 \frac{\sqrt{n}\ (n - \epsilon(s))}{8};
\ee 
\item[\rm (c)] 
if $\sqrt{n} + 1 \le s  < n/2$, then 
\be
\label{eq:tf3}
\max\left\{\frac{n^2 - \epsilon(n)}{8(s+1)},\frac{(n-1)(\sqrt{2n}-7)^3}{24n} \right\}
 \leq
 \frac{\pi(C_n(1,s))}{2} \le \api(C_n(1,s)) $$ $$ \hspace{3in}
\leq
 \frac{s^2 (n + r+ 2) - \epsilon(s) (n-r)}{8s}.
\ee 
\end{itemize}
\end{thm}

Let $\mathtt{ratio}$ denote the ratio of the upper bound to the lower bound in the same equation above. In (\ref{eq:tf1}), $\mathtt{ratio} =   \frac{s+1}{s} \left(\left(1-\frac{r}{n}\right)\left(1+\frac{r+2}{n}\right) + \frac{s^2}{n^2}\right)\frac{n^2}{n^2 - \epsilon(n)}$, which is asymptotically $(s+1)/s$ ($\le 3/2$) as $n \rightarrow \infty$. 
In (\ref{eq:tf2}), $\mathtt{ratio} = \frac{n-\epsilon(n)}{n-1}$, which tends to 1 as $n \rightarrow \infty$. 
By using the lower bound $(n^2-\epsilon(n))/8(s+1)$ and the upper bound  $s(n+r+2)/8$ in (\ref{eq:tf3}), we have $\mathtt{ratio} \le (s^2+s)$$ (n+r+2) / (n^2 - \epsilon(n))$. If $ s = c\sqrt{n}$, for $c\in(1,\sqrt{3/2}]$, then $\mathtt{ratio} \le  c^2 +o(1)$, which is at most  $3/2$ asymptotically.     
If $\sqrt{3n/2} < s < n/2$, then the lower bound in (\ref{eq:tf3}) is equal to $(n-1)(\sqrt{2n}-7)^3/24n$ and $\mathtt{ratio} \le 3s n (n+1)/(n-1)(\sqrt{2n}-7)^3$. In the latter case, the upper bound increases with $s$ and can be $O(\sqrt{n})$ in the worst case scenario when $s \approx cn$ for some constant $c < 1/2$.  
 
We will prove the lower bounds in (\ref{eq:tf1})-(\ref{eq:tf3}) in the next section. The upper bounds will be proved in Section \ref{sec:ub}; see Lemma \ref{lem:ub1} which gives more information and better upper bounds in some cases. To establish the upper bounds, we will give a specific routing (see Construction \ref{const}), which can be viewed as an approximation algorithm for computing $\api(C_n(1,s)$ and $\pi(C_n(1,s))$. From the discussion above, for $2 \le  s \le \sqrt{3n/2}$, the performance ratio of this algorithm is at most  $3/2$ asymptotically. So we obtain the following corollary of Theorem \ref{thm:tr1}.

\begin{cor}
\label{cor:tr1}
There is a 1.5-factor approximation algorithm to solve the edge-forwarding and arc-forwarding problems for 4-regular circulant graphs $C_n(1, s)$ with $n$ sufficiently large and $2 \le s \le \sqrt{3n/2}$. 
\end{cor} 

In the worst case when $s \approx cn$ is large, where $c < 1/2$ is a constant, the ratio obtained from (\ref{eq:tf3}) (and that of the approximation algorithm from Construction \ref{const}) is $O(\sqrt{n})$. It seems that this large ratio is due to the fact that the lower bound in (\ref{eq:tf3}) is unsatisfactory when $s$ is large. Our next result shows that sometimes we can significantly improve this lower bound for large $s$. This enhanced lower bound together with the upper bound in (\ref{eq:tf3}) implies that for $s \approx c n$ our algorithm can achieve a constant performance ratio $1/c$ in some cases. 
 
\begin{thm} 
\label{thm:lwr3}
If $r\leq q$ or $r+q\ge s+1$, then
\be
\label{eq:lwr3}
\pi(C_n(1,s))\geq \frac{1}{2} \left\lfloor\frac{(s+1)^2}{2}\right\rfloor.
\ee
\end{thm}
We prove Theorem \ref{thm:lwr3} by computing the sum of the distances between all pairs of nodes in $C_n(1,s)$, which is done by investigating an equivalent problem \cite{Zerovnik1993} for the integer lattice $\ZZ^2$.

The second main result in this paper is the following theorem on the optical indices of $C_n(1,s)$.   Denote
\be
\label{eq:ka}
\kappa(a) := a+\frac{\epsilon(s)+\epsilon(q)}{2}.
\ee

\begin{thm}
\label{thm:tr2}
The following hold:
 \begin{itemize} 
\item[\rm (a)]  if $2 \le s \le \sqrt{n-r + (\kappa(-2))^2} + \kappa(-2)$, then
\be
\label{eq:thw1}
    \frac{n^2 - \epsilon(n)}{8(s+1)} \leq  
\aw(C_n(1,s)) \leq 
\frac{s+2}{24}\left(6q^2   + 3 q (s+4) +  s (4s+10) +  \epsilon(q) (2q + 3s +3)\right)  ; 
\ee

\item[\rm (b)]  if $ \sqrt{n-r+(\kappa(-1))^2} +\kappa(-1) \le s \le \sqrt{n-r + (\kappa(0))^2} + \kappa(0)$, then
\be
\label{eq:thw2}
\frac{n^2 - \epsilon(n)}{8(s+1)} \leq  \aw(C_n(1,s))\le
\frac{ q(q+2)( 5q+2)}{24} + \frac{s (s+2) (2s+5)}{6}+\epsilon(q)\frac{5 q^2 + 13 q + 7}{8};
\ee

\item[\rm (c)] if $\sqrt{n-r+(\kappa(1))^2}+ \kappa(1) \le s  < n/2$, then 
\be
\label{eq:thw3}
\max\left\{ \frac{n^2 - \epsilon(n)}{8(s+1)},\frac{(n-1)(\sqrt{2n}-7)^3}{24n} \right\}
 \leq  
\aw(C_n(1,s)) \hspace{1in} $$ $$ \hspace{1in}\le
 \frac{q(q+2) (q+10)}{24} + \frac{s(s+2)(q+1)}{2}  +\epsilon(q)\frac{(q+5)^2 + 4 (s+1)^2 }{8}. 
\ee
\end{itemize}
Moreover, the same lower and upper bounds are valid if $\aw(C_n(1,s))$ is replaced by $w(C_n(1,s))/2$ in (\ref{eq:thw1})-(\ref{eq:thw3}). 
\end{thm}

Let $\mathsf{ratio}$ denote the ratio of the upper bound to the lower bound in the same equation above.
In (\ref{eq:thw1}), $\mathsf{ratio} \le    2(1+\frac{1}{s})(1+\frac{2}{s}) \left(1+ \frac{s^2}{  2n} + \frac{2s^4}{3 n^2} \right)\frac{n^2}{n^2-\epsilon(n)}+O\left(\frac{1}{\sqrt{n}}\right)$, which is $2(1+\frac{1}{s})(1+\frac{2}{s})$ asymptotically when $s^2 = o(n)$. When $s^2 = \Omega(n)$, in (\ref{eq:thw1}) we have $\mathsf{ratio} \le 13/3$  asymptotically since $s \le \sqrt{n-r + (\kappa(-2))^2} + \kappa(-2)$. 
In (\ref{eq:thw2}),  $\mathsf{ratio} \le 13/3$  asymptotically  as $s\approx q \approx \sqrt{n}$ in case (b).   Let  
\be
\label{eq:delta}
\delta(n) := \smid-1. 
\ee
  Then $\delta(n) + 1 > 3\sqrt{n}/2\sqrt{2} > 1$ if $n \ge 25$, $\delta(n) \approx 3\sqrt{n}/2\sqrt{2}$ as $n \rightarrow \infty$, and $\delta(n) +1 \le 91\sqrt{n}/80$ for sufficiently large $n$.  
If $s \le \delta(n)$, then the lower bound in (\ref{eq:thw3}) is equal to $(n^2 - \epsilon(n))/8(s+1)$ and $\mathtt{ratio}\le \frac{q (q+2)(q+10)(s+1)}{3 (n^2-\epsilon(n))} + \frac{4 (s+1) (s+2)(n-r+s)}{ n^2-\epsilon(n)} + O(\frac{1}{\sqrt{n}})$. This upper bound increases with $s$ and approaches a constant between $13/3$ and $4.85$, depending on the value of $s$, as $n \rightarrow \infty$. The upper bounds in (\ref{eq:thw1})-(\ref{eq:thw3}) will be proved by giving a specific colouring (see Construction \ref{colouring}) of the routing obtained from Construction \ref{const}, and this can be viewed as an approximation algorithm for the routing and wavelength assignment problem and its oriented version for $C_n(1, s)$. Thus the discussion above yields the following corollary of Theorem \ref{thm:tr2}.

\begin{cor}
\label{cor:tr2}
There is a $4.85$-factor approximation algorithm to solve the routing and wavelength assignment problem and its oriented version for 4-regular circulant graphs $C_n(1, s)$ with $n$ sufficiently large and $3 \le s \le  3\sqrt{n}/2\sqrt{2} - 1$.
\end{cor} 

If $\delta(n) < s < n/2$, then the lower bound in (\ref{eq:thw3}) is equal to $(n-1)(\sqrt{2n}-7)^3/24n$ and the $\mathsf{ratio}$ from (\ref{eq:thw3}) (and so the approximation ratio of the algorithm from Construction \ref{colouring}) is at most $\frac{n} {(n - 1)}\left(\frac{ q(q+2) (q+10)}{(\sqrt{2n}-7 )^3} + \frac{s(s+2)(q+1)+\epsilon(q) s^2 }{2(\sqrt{2n}-7 )^3}\right)+ O(\frac{1}{\sqrt{n}})$. This upper bound increases with $s$ and is $O(\sqrt{n})$ in the worst case scenario when $s \approx cn$ for some constant $c < 1/2$. However, in the case when $r \le q$ or $r+q \ge s+1$, our stronger lower bound on $w(C_n(1,s))$ obtained from (\ref{eq:lwr3}) and (\ref{eq:piw}) implies that the approximation ratio is at most $4(1/c+ \epsilon(q))$ asymptotically. 

The rest of the paper is organised as follows. 	 
We will prove the lower bounds in the theorems above in the next section. In Section \ref{sec:ub} we will establish the upper bounds in Theorem \ref{thm:tr1} by devising a specific routing (Construction \ref{const}). In Section \ref{sec:wave} we will prove the upper bounds in Theorem \ref{thm:tr2} by giving a specific colouring (Construction \ref{colouring}) for the routing obtained from Construction \ref{const}. Note that the lower bounds in (\ref{eq:thw1})-(\ref{eq:thw3}) are obtained from (\ref{eq:piw}) and the lower bounds in Theorem \ref{thm:tr1}.

\section{Lower bounds}
\label{sec:lwrs}

\subsection{Two lower bounds}
\label{subsec:two}

Given a graph $G$ and $U \subset V(G)$, let $\delta(U)$ denote the set of edges of $G$ with one end in $U$ and the other end in $\overline{U} = V(G) \setminus U$. Let $R^*$ be a routing of $G$ such that $\pi(G)=\pi(G, R^*)$. Then $\pi(G)$ is the maximum load on an edge of $G$ under $R^*$. The total load on the edges of $\delta(U)$ under $R^*$ is thus at most $\pi(G)|\delta(U)|$. On the other hand, there are exactly $2 |U| |\overline{U}|$ paths in $R^*$ with one end in $U$ and the other end in $\overline{U}$, and each of them uses at least one edge of $\delta(U)$. Therefore, 
\begin{equation}
\label{eq:is2}
\pi(G)|\delta(U)| \geq 2 |U| |\overline{U}|.
\end{equation}  
 	 
\begin{lem}
\label{lem:lwr1}
$$
\pi(C_n(1,s)) \ge \frac{\lfloor n/2\rfloor\lceil n/2\rceil}{s+1} = \frac{n^2 - \epsilon(n)}{4(s+1)}.
$$ 
\end{lem}

\begin{proof}
We apply (\ref{eq:is2}) to $C_n(1,s)$. Choose $U=\{0,\dots,\lfloor n/2\rfloor-1\}\subset\mathbb{Z}_n$ so that $|U|=\lfloor n/2\rfloor$ and $|\overline{U}|=\lceil n/2\rceil $. Consider the neighbours $i+s$, $i-s$ of $i \in U$. We have: $\lfloor n/2\rfloor \le i+s \le n-1$ if and only if $\lfloor n/2\rfloor-s\le i\le \lfloor n/2\rfloor-1$, and $i - s$ ($\equiv n-(s-i) \mod n$) lies in $\overline{U}$ if and only if $0\le i\le s-1$. Thus $\delta(U)$ consists of edges $\{i, i+s\}$ ($\lfloor n/2\rfloor-s\le i\le \lfloor n/2\rfloor-1$), $\{i,i-s\}$ ($0\le i\le s-1$), $\{0,n-1\}$ and $\left\{\lfloor n/2\rfloor-1,\lfloor n/2\rfloor\right\}$. Hence $|\delta(U)|=2s+2$. This together with (\ref{eq:is2}) yields $\pi\big(C_n(1,s)\big) \geq \lfloor n/2\rfloor\lceil n/2\rceil/(s+1)$. 
\end{proof}

As observed in  \cite{Heydemann1989}, we have 
	\begin{equation}
	\label{eq:trvl}
	\pi(G)\geq\frac{\sum_{(x, y)\in V(G) \times V(G)}d(x, y)}{|E(G)|} = \frac{n(n-1)\overline{d}(G)}{|E(G)|},
	\end{equation}
where $d(x, y)$ is the distance between $x$ and $y$ in $G$ and $\overline{d}(G)$ is the mean distance among all unordered pairs of vertices of $G$. It was proved in \cite[Theorem 4.6]{Hwang2003} that, if $G$ is a circulant network of order $n$ and degree $4$, then
$$
\overline{d}(G) \geq\frac{(\sqrt{2n}-7)^3}{6n}.
$$
Applying this and (\ref{eq:trvl}) to $C_n(1,s)$, we obtain:
	
	\begin{lem}
	\label{lem:lwr2}
$$
\pi(C_n(1,s))\geq \frac{(n-1)(\sqrt{2n}-7)^3}{12n}.
$$
\end{lem}

It can be verified that the lower bound in Lemma \ref{lem:lwr1} is no less than that in Lemma \ref{lem:lwr2} if and only if $s \le 3n(n^2 - \epsilon(n))/(n-1)(\sqrt{2n}-7)^3 -1$. Thus the lower bounds in (\ref{eq:tf1})-(\ref{eq:tf3}) follow from (\ref{eq:pi}) and Lemmas \ref{lem:lwr1} and \ref{lem:lwr2} immediately.

\subsection{Proof of Theorem \ref{thm:lwr3}}
\label{subsec:rq}

To prove Theorem \ref{thm:lwr3} we need two results from \cite{Zerovnik1993}. 
Let $\ZZ^2$ be the $2$-dimensional $\ZZ$-module lattice. Define $l:\ZZ^2 \rightarrow \ZZ_n$ by 
$$
l(\mathbf{x}):=x_1\oplus x_2 s,\;\, \mathbf{x}=(x_1,x_2) \in \ZZ^2.
$$
We may view $l$ as a labelling that labels each point $(x_1, x_2)$ of $\ZZ^2$ by the node $x_1 \oplus x_2 s$ of $C_n(1,s)$. We observe that when two points $(x_1,x_2)$ and $(y_1,y_2)$ are neighbours in the lattice $\ZZ^2$ (that is, either $x_1=y_1$ and $|x_2-y_2|=1$, or $|x_1-y_1|=1$ and $x_2=y_2$), the corresponding labels $l(x_1,x_2)$ and $l(y_1,y_2)$ are adjacent nodes of $C_n(1, s)$. 

Denote by $\norm{.}$ the $L_1$-norm in $\ZZ^2$ defined by 
$$
{\norm{\mathbf{x}}} := |x_1|+|x_2|,\;\, \mathbf{x}=(x_1,x_2) \in \ZZ^2.
$$ 
The {\em length} of a path $\mathbf{x}_0, \mathbf{x}_1,\dots,\mathbf{x}_k$ in $\ZZ^2$,  connecting $\mathbf{x}_0$ and $\mathbf{x}_k$ is defined as $k$, and the {\em distance} between two points of  $\ZZ^2$ is defined to be the length of a shortest path in $\ZZ^2$ connecting them, where $\mathbf{x}_{i-1}$ and $\mathbf{x}_{i}$ are neighbours in the lattice. Thus the distance between $\mathbf{x}$ and $\mathbf{y}$ in $\ZZ^2$ is equal to $\norm{\mathbf{x}-\mathbf{y}}$ \cite{Zerovnik1993}. 

Each path $\mathbf{x}_0, \mathbf{x}_1,\dots,\mathbf{x}_k$ in $\ZZ^2$ gives rise to the oriented path $l(\mathbf{x}_0), l(\mathbf{x}_1), \ldots, l(\mathbf{x}_k)$ in $C_n(1, s)$. Note that even if the former is a shortest path in $\ZZ^2$, the latter is not necessarily a shortest path in $C_n(1, s)$. 

Denote by $X$ the set of points of $\ZZ^2$ with label $0$. That is, 
$$
X:=\left\{(x_1,x_2)\in\ZZ^2: l(x_1,x_2) = 0\right\}.
$$  
Note that $X$ relies on $s$ implicitly. A \emph{basis} for $X$ is a set of two independent vectors $\left\{\mathbf{a},\mathbf{b}\right\}$ in $X$ such that any vector in $X$ is a linear combination of them with integer coefficients. The {\em parallelogram} \cite{Zerovnik1993} generated by a basis $\{\mathbf{a},\mathbf{b}\}$ is defined as  
$$
[\mathbf{a},\mathbf{b}] := \left\{\mathbf{x} \in\ZZ^2: \mathbf{x}=\alpha \mathbf{a}+\beta\mathbf{b},0\leq\alpha,\beta \le 1\right\}.
$$ 
Note that its corner points, $\mathbf{0}, \mathbf{a},\mathbf{b}$ and $\mathbf{a+b}$, are in $X$. 
Similarly, the {\em half-open parallelogram} generated by $\{\mathbf{a},\mathbf{b}\}$ is defined as
$$
[\mathbf{a},\mathbf{b}):=\left\{\mathbf{x} \in\ZZ^2: \mathbf{x}=\alpha \mathbf{a}+\beta\mathbf{b},0\leq\alpha,\beta<1\right\}.
$$   
In what follows we use $d(i, j)$ to denote the distance in $C_n(1, s)$ between $i \in \ZZ_n$ and $j \in \ZZ_n$. It is observed in \cite{Zerovnik1993} that, for every $\mathbf{v}\in\ZZ^2$, we have 
\be
\label{eq:dl}
d\big(0,l(\mathbf{v})\big) := \norm{\mathbf{v}-X}:=\min\left\{\norm{\mathbf{v-x}}: \mathbf{x}\in X\right\}.
\ee

\begin{lem} (\cite[Proposition 2]{Zerovnik1993})
\label{prop:conta}
Let $\{\mathbf{a},\mathbf{b}\}$ be a basis for $X$. Then $[\mathbf{a},\mathbf{b})$ has exactly $n$ points and each label in $\{0,1,\dots,n-1\}$ appears exactly once as $l(\mathbf{x})$ for some $\mathbf{x} \in [\mathbf{a},\mathbf{b})$.
\end{lem}

Thus the labelling $l$ induces a bijection between the points in $[\mathbf{a},\mathbf{b})$ and the nodes in $\ZZ_n$. This together with (\ref{eq:dl}) implies that for any $i \in \ZZ_n$,  
$d(0, i) = \norm{\mathbf{v}_i-X}$, where $\mathbf{v}_i$ is the unique point in $[\mathbf{a},\mathbf{b})$ with $l(\mathbf{v}_i) = i$. The next lemma says that, if $\{\mathbf{a},\mathbf{b}\}$ is a packed basis, then $\norm{\mathbf{v}_i-X}$ is attained at a corner point of $[\mathbf{a},\mathbf{b}]$, where a basis $\{\mathbf{a,b}\}$ is \emph{packed}
\cite{Zerovnik1993} if it satisfies 
\be
\label{eq:packed}
\max\{\norm{\mathbf{a}},\norm{\mathbf{b}}\}\leq\min\{\norm{\mathbf{a-b}},\norm{\mathbf{a+b}}\}.
\ee   

\begin{lem}  
(\cite[Lemma 2]{Zerovnik1993})
\label{lem:dist}
Let $\{\mathbf{a},\mathbf{b}\}$ be a packed basis for $X$. Then, for any $\mathbf{v} \in [\mathbf{a},\mathbf{b}]$, we have
$$
\norm{\mathbf{v} - X} = \min 
\{{\norm{\mathbf{v} - \mathbf{x}}: \mathbf{x}} = \mathbf{0}, \mathbf{a}, \mathbf{b}, \mathbf{a} + \mathbf{b}\}.
$$
\end{lem} 

\begin{lem}
\label{lem:pb} 
The following hold:
\begin{itemize}
\item[\rm (a)] if $r\le q$ and $2r\le s + 1$, then $\{(s,-1),(r,q)\}$ is a packed basis for $X$;
\item[\rm (b)] if $r\le q$ and $2r\ge s + 1$, then $\{(s,-1),(r-s, q+1)\}$ is a packed basis for $X$;
\item[\rm (c)] if $r\ge q$ and $r + q \ge s + 1$, then $\{(s,-1),(r-s,q+1)\}$ is a packed basis for $X$. 
\end{itemize}
\end{lem}

\begin{proof}
It can be shown (see \cite[pp.6]{Chen2005}) that any point in $X$ is of the form $i(s,-1)+j(r,q)$ for some integers $i$ and $j$. (In fact, for any $(x_1, x_2) \in X$, we have $x_1+x_2 s=kn$ for some integer $k$, and so $(x_1, x_2)=(kq-x_2)(s,-1)+k(r,q)$.) It follows that the pair $\{\mathbf{a,b}\}$ in each case is a basis for $X$. It remains to verify that $\{\ba, \bb\}$ satisfies (\ref{eq:packed}). Recall that $q \ge 2$ as $s < n/2$.
 
(a) Let $\ba = (s, -1)$ and $\bb = (r, q)$. Since $r\le q$, $2r\le s + 1$ and $q \ge 2$, we have $\norm{\ba} = s+1, \norm{\bb} = r+q, \norm{\ba - \bb} = s - r +  q + 1$, $\norm{\ba+\bb} = s + r +  q - 1$, and $\{\ba, \bb\}$ satisfies (\ref{eq:packed}).

(b) Let $\ba = (s, -1)$ and $\bb = (r - s, q + 1)$. Since $r\le q$ and $2r\ge s + 1$, we have $\norm{\ba} = s+1, \norm{\bb} = s - r+q+1, \norm{\ba - \bb} = 2s - r +  q + 2$ and $\norm{\ba+\bb} = r +  q$. Hence $\norm{\ba} \le \norm{\bb} \le \norm{\ba + \bb} \le \norm{\ba - \bb}$ and $\{\ba, \bb\}$ satisfies (\ref{eq:packed}).  

(c) Let $\ba$ and $\bb$ be as in (b). Since $r\ge q$ and $r + q \ge s + 1$, the norms of $\ba, \bb, \ba - \bb, \ba+\bb$ are the same as in (b). Hence $\norm{\bb} \le \norm{\ba} \le \norm{\ba + \bb} \le \norm{\ba - \bb}$ and $\{\ba, \bb\}$ satisfies (\ref{eq:packed}).  
\end{proof}

	\begin{lem}
	\label{lem:Ldist}
If $r\le q$ or $r + q\ge s + 1$, then
\be 
\label{eq:lb}
\sum_{i\in\ZZ_n}d(0,i)\geq \left\lfloor\frac{(s+1)^2}{2}\right\rfloor.
\ee
	\end{lem}
	
\begin{proof}
Since $r\le q$ or $r + q\ge s + 1$, one of the three cases in Lemma \ref{lem:pb} occurs, and in each case we have a packed basis $\{\ba, \bb\}$ for $X$ as given in Lemma \ref{lem:pb}. By Lemma \ref{prop:conta}, for any $i \in \ZZ_n$, there exists at least one $\mathbf{v}\in[\mathbf{a,b}]$ such that $l(\mathbf{v})=i$. Moreover, by Lemma \ref{lem:dist}, $d(0,i)=\min\{{\norm{\mathbf{v} - \mathbf{x}}: \mathbf{x}} = \mf{0}, \mf{a}, \mf{b}, \mf{a} + \mf{b}\}$. We now compute the sum of these distances $d(0,i)$ for $i$ in a certain subset of $\ZZ_n$. 

\medskip\textsf{Case 1:}~$r\le q$ and $2r\le s+1$. In this case we have $\ba=(s,-1),\bb=(r,q)$ by Lemma \ref{lem:pb}. Set $\bv_i=(i,0)$, $\bw_i=(\ba+\bb)-(i,0)$, $\alpha_i=iq/n$, $\beta_i=i/n$, $\alpha_i'=1-iq/n$ and $\beta_i'=1-i/n$ for $1 \le i \le s$. Then $\alpha_i,\beta_i,\alpha_i',\beta_i'\in[0,1]$, $\bv_i=\alpha_i\ba+\beta_i\bb$, $\bw_i=\alpha_i'\ba+\beta_i'\bb$ and $\bv_i, \bw_i\in[\ba, \bb]$ for each $i$. Since $l(\bv_i)=i,l(\bw_i)=n-i$ and  $s< n/2$, $l(\bv_1), \ldots, l(\bv_s), l(\bw_1), \ldots, l(\bw_s)$  are pairwise distinct. It can be verified that
$$
\norm{\bv_i - \b0}=\norm{\bw_i-(\ba+\bb)}=i,\quad \norm{\bv_i-\ba}=\norm{\bw_i-\bb}=s-i+1,
$$ 
$$
\norm{\bv_i-(\ba+\bb)}=\norm{\bw_i-\b0}=s+r-i+q-1,\quad \norm{\bv_i-\bb}=\norm{\bw_i-\ba}=|r-i|+q.
$$ 
Assume $1\leq i\leq \lfloor s/2\rfloor$. Since $r\leq q$, we have $|r-i| + q\ge i$ (as $i-r+q\ge i$ if $i>r$ and $(r-i)+q\ge r\ge i$ if $i\le r$) and $s+r-i+q-1 \ge s + 1 - i \ge i$. Therefore, $\norm{\bv_i-X}=\norm{\bv_i-\b0}$ and $\norm{\bw_i-X}=\norm{\bw_i-(\ba+\bb)}$. For $\lfloor s/2\rfloor< i\leq s$, it can be verified that 
$\norm{\bv_i-X}=\norm{\bv_i-\ba}$ and $\norm{\bw_i-X}=\norm{\bw_i-\bb}$.
Therefore, 
$$
\begin{array}{lll}
\sum_{i\in\ZZ_n}d(0,i) & \geq & \sum_{i=1}^{s}(\norm{\bv_i-X}+\norm{\bw_i-X}) \\[0.2cm]
& = & \sum_{i=1}^{\lfloor s/2 \rfloor} 2i + \sum_{\lfloor s/2 \rfloor + 1}^s 2(s-i+1) \\[0.2cm]
& \ge & \big \lfloor\frac{(s+1)^2}{2} \big \rfloor.
\end{array}
$$

\medskip
\textsf{Case 2:}~either $r\le q$ and $2r\ge s+1$, or $r\ge q$ and $r+q\ge s+1$. Then $r+q\ge s+1$,  $2r\ge s+1$, and $\ba=(s,-1), \bb=(r-s,q+1)$ by Lemma \ref{lem:pb}.
 Set $\bv_i=(i,0)$, $\bw_i=(\ba+\bb)-(i,0)$, $\alpha_i=i(q+1)/n$, $\beta_i=i/n$, $\alpha_i'=1-i(q+1)/n$, $\beta_i'=1-(i/n)$ for $1\le i\le s-1$ and $\bu=(1,1)$.  Then $\alpha_i,\beta_i,\alpha_i',\beta_i'\in[0,1]$, $\bv_i=\alpha_i\ba+\beta_i\bb$, $\bw_i=\alpha_i'\ba+\beta_i'\bb$, and $\bv_i, \bw_i \in[\ba, \bb]$ for each $i$. Moreover, $\bu=((s-r+q+1)/n)\ba+((s+1)/n)\bb \in[\ba, \bb]$. Since  $l(\bv_i)=i,\l(\bw_i)=n-i,l(\bu)=s+1$ and $s<n/2$, $l(\bv_1),\dots, l(\bv_{s-1}), l(\bw_1), \dots, l(\bw_{s-1}) ,l(\bu)$ are pairwise distinct.
	It can be verified that 
	$$\norm{\bv_i-\b0}=\norm{\mathbf{w}_i-(\ba+\bb)}=i,\quad \norm{\bv_i-\ba}=\norm{\bw_i-\bb}=s-i+1,$$
	$$\norm{\bv_i-(\ba+\bb)}=\norm{\bw_i-\b0}=|r-i|+q, \quad \norm{\bv_i-\bb}=\norm{\bw_i-\ba}=s-r+i+q+1.$$

Assume $1\leq i\leq \lfloor s/2\rfloor$. Since $r+q\ge s+1$, we have  $ r-i+q\ge s-i+1\ge i$ and $ s-r+i+q+1\ge i$. Hence $\norm{\mathbf{v}_i-X}=\norm{\mathbf{v}_i-\b0}$ and $\norm{\bw_i-X}=\norm{\bw_i-(\ba+\bb)}$. If $\lfloor s/2\rfloor< i\leq s-1$, then $\norm{\mathbf{v}_i-X}=\norm{\mathbf{v}_i-\ba}$ and $\norm{\mathbf{w}_i-X}=\norm{\mathbf{w}_i-\bb}$. Note that $\norm{\bu-X}=2.$
		Therefore, 
	$$
\begin{array}{lll}
\sum_{i\in\ZZ_n}d(0,i) & \geq & \sum_{i=1}^{s-1}(\norm{\mathbf{v}_i-X}+\norm{\mathbf{w}_i-X})+\norm{\bu-X} \\ [0.2cm]
& = & \sum_{i=1}^{\lfloor s/2 \rfloor} 2i + \sum_{\lfloor s/2 \rfloor + 1}^{s-1} 2(s-i+1) + 2 \\[0.2cm]
& \ge & \big\lfloor \frac{(s+1)^2}{2}\big\rfloor.
\end{array}
$$
\end{proof}	

\begin{proof}[Proof of Theorem \ref{thm:lwr3}]
Suppose $r \le q$ or $r + q \ge s + 1$. Since $C_{n}(1, s)$ is vertex-transitive, we have $\sum_{(i, j) \in \ZZ_n \times \ZZ_n} d(i, j) = n \sum_{i \in \ZZ_n} d(0, j)$. This together with (\ref{eq:trvl}) and (\ref{eq:lb}) implies (\ref{eq:lwr3}). 
\end{proof}

As shown in the proof of Lemma \ref{lem:Ldist}, we obtained (\ref{eq:lb}) by computing the total distance from the vertex $0$ to $2s$ or $2s-1$ other vertices. This implies that the difference between the two sides of (\ref{eq:lb}) is small if and only if $s$ is close to $n/2$.  

The sum of the distances can be precisely computed when $s= \sqrt{n}$, leading to the following better lower bound than Theorem \ref{thm:lwr3} in this special case.  
\begin{lem}
	\label{lem:ssqrtn}
	If $s= q=\sqrt{n}$, then
	\[
			\pi(C_n(1,s)) \ge   \frac{\sqrt{n}(n-1)}{4}.
	\]
\end{lem}
\begin{proof}
	By Lemma \ref{lem:pb}, $\{ (\sqrt{n},-1), (0,\sqrt{n})\}$ is a packed basis for $X$ in this case. One can see that the following is a closest corner point to $(i,j)$: 
	\begin{itemize}
		\item[\rm (a)] $(0,0)$, if $0\le i \le \lfloor \sqrt{n}/2 \rfloor$ and $0\le j\le \lfloor (\sqrt{n}-1)/2 \rfloor$;
		\item[\rm (b)] $(\sqrt{n},-1)$, if $\lfloor  \sqrt{n}/2 \rfloor +1  \le i\le  \sqrt{n}-1  $ and $0\le j \le \lfloor  \sqrt{n} /2 \rfloor-1$;
		\item[\rm (c)]  $(0,\sqrt{n})$, if $0 \le i \le \lfloor (\sqrt{n}-1)/2 \rfloor$ and $\lfloor (\sqrt{n}+ 1)/2 \rfloor \le j \le \sqrt{n} -1$ and 
		\item[\rm (d)]  $( \sqrt{n},\sqrt{n}-1)$, if $\lfloor (\sqrt{n}+1)/2 \rfloor \le i \le \sqrt{n}-1$ and $\lfloor  \sqrt{n}/2 \rfloor   \le j \le \sqrt{n} -1$.
	\end{itemize}
Since every point $(i, j)$ in $X$ appears in exactly one of these four cases, we have 
	\[
	\sum_{k\in\ZZ_n}d(0,k) = \sum_{i=0}^{\lfloor \sqrt{n}/2 \rfloor}\sum_{j=0}^{\lfloor (\sqrt{n}-1)/2 \rfloor} (i+j) +  \sum_{i=\lfloor \sqrt{n}/2 \rfloor+1}^{\sqrt{n}-1 } \sum_{j=0}^{\lfloor \sqrt{n}/2 \rfloor-1} (\sqrt{n}-i+j+1) +
	\]
	\[
	\sum_{i=0}^{\lfloor (\sqrt{n}-1)/2 \rfloor}\sum_{j=\lfloor (\sqrt{n}+1)/2 \rfloor}^{\sqrt{n}-1 } (i+\sqrt{n}-j) + \sum_{i=\lfloor (\sqrt{n}+1)/2 \rfloor}^{\sqrt{n}-1 } \sum_{j=\lfloor  \sqrt{n} /2 \rfloor}^{\sqrt{n}-1 } (\sqrt{n}-i+\sqrt{n}-1-j)= \frac{\sqrt{n}(n-1)}{2}.
	\]
The result then follows from (\ref{eq:trvl}) and vertex-transitivity of $ C_n(1,s)$. 
\end{proof}

\section{A routing scheme, and proof of Theorem \ref{thm:tr1}}
\label{sec:ub}

In this section we give a specific routing scheme for $C_n(1,s)$ which yields the required upper bounds on the forwarding indices of $C_n(1,s)$. The same routing will be used in the next section to give upper bounds on the optical indices of $C_n(1,s)$. We will use the words `link' and `arc' interchangeably and we call a link of $C_n(1,s)$ of the form $(x,x\oplus1)$ ($(x,x\ominus 1)$, respectively) a {\em clockwise} ({\em anticlockwise}, respectively) {\em ring link}, and a link of the form $(x,x\oplus s)$ ($(x,x\ominus s)$, respectively) a {\em clockwise} ({\em anticlockwise}, respectively) {\em skip link}. We define a routing as follows.

\begin{const}
\label{const}
{\em 
Define 
\be
\label{eq:R}
{\cal R} := \{P_{x, y}: x, y \in \ZZ_n, x \ne y\},
\ee
where $P_{x, y}$ is the path in $C_n(1,s)$ from $x$ to $y$ specified as follows. 
\begin{enumerate} 
	\item \label{step:s1}
		For $d = 1, \ldots, \nn$, say, $d=is+j$ for some $i,j$ with $0\le i\le \qq$ and $0\le j\le s-1$, 
	\begin{enumerate}[(a)]
		\item if $j \leq \ss $, then define $P_{0,d}: 0,s,2s,\dots,is,is+1, is + 2, \dots,is+j$;
		\item if $j> \ss $, then define $P_{0,d}: 0,s,2s,\dots,(i+1)s,(i+1)s-1,(i+1)s-2,\dots,(i+1)s-(s-j)$. 
	\end{enumerate}

	\item \label{step:s3}
		For $d = \nn +1, \ldots, n-1$, letting $P_{0, n-d}: v_1,v_2,\dots,v_k$ be the path from $0$ to $n-d$ constructed in Step 1, define $P_{0,d}: v_1, n-v_2,\dots, n-v_k$. 

	\item \label{step:s4}
		For $1\leq x, y\leq n-1$ with $x \ne y$, letting $v_1,v_2,\dots,v_k$ denote the path $P_{0,y\ominus x}$ from $0$ to $y\ominus x$ constructed in Step 1 or 2, 
define $P_{x,y}: x\oplus v_1,x\oplus v_2,\dots, x\oplus v_k$.
\end{enumerate}
}
\end{const}

The routing $\cal R$ constructed above is symmetric in the sense that, for any $x, y, k \in \ZZ_n$ with $x \ne y$, $P_{x\oplus k, y\oplus k}$ is the path obtained by adding $k$ to each node of $P_{x, y}$.  We say that $P_{x\oplus k, y\oplus k}$ is obtained from {\em translation} of $P_{x, y}$ by $k$. This feature is crucial in the following computation of $\api(C_n(1,s),{\cal R})$. Denote
$$
\Delta := \left\{ 
\begin{array}{ll}
\frac{s}{4} + \frac{1}{2} \left\lfloor\frac{r}{2}\right\rfloor \left(s - \left\lfloor\frac{r+2}{2}\right\rfloor\right),\;\; \mbox{if $s$ is even}\\ [0.3cm]
\frac{1}{2} \left\lfloor\frac{r+1}{2}\right\rfloor \left(s - \left\lfloor\frac{r+1}{2}\right\rfloor\right),\;\; \mbox{if $s$ is odd}.
\end{array}
\right.
$$
One can verify that the four numbers involved in the next lemma are integers.

\begin{lem}
\label{lem:ub}
\begin{itemize} 
\item[\rm (a)] If $q$ is even, then  
$$ 
\api(C_n(1,s),{\cal R}) = 
\max\left\{  \frac{q}{4} \BL{s^2}{2}  + \frac{1}{2}\brrr \BL{r+2}{2},
\frac{q^2s}{8} + \frac{q}{2} \left(\brrr + \frac{\epsilon(s)  }{2}\right)  \right\}.
$$
\item[\rm (b)] If $q$ is odd, then
$$
\api(C_n(1,s),{\cal R}) =
\max\left\{  \frac{q}{4}\BL{s^2}{2} +\Delta, 
\frac{(q^2-1)s}{8} + \frac{q+1}{2}\left\lfloor\frac{r+\epsilon(s)}{2}\right\rfloor
\right\}.
$$
\end{itemize}
\end{lem} 

\begin{proof}
Observe that if a path $P_{x, y}$ in $\cal R$ passes through a clockwise ring link $(v, v\oplus1)$, then the path $P_{x\oplus (v'- v), y\oplus (v' - v)}$ (which is also in $\cal R$) passes through the clockwise ring link $(v', v'\oplus 1)$. Hence the loads on all clockwise ring links under $\cal R$ are equal. Similarly, all anticlockwise ring links have the same load, all clockwise skip links have the same load, and all anticlockwise skip links have the same load. So the load on each clockwise ring link (skip link, respectively) is the total number of clockwise ring links (skip links, respectively) used by paths of $\cal R$ divided by $n$. The same can be said for anticlockwise ring or skip links. 

On the other hand, for any fixed $x\in\ZZ_n$, the paths $P_{x,y}$, $x \ne y=0,\dots,n-1$, use the same number of ring (skip, respectively) links as the paths $P_{0,d}$, $d=1,\dots,n-1$,  because $P_{x,y}$ is the translation of $P_{0,y\ominus x}$ by $x$ and therefore there is a bijection between these two sets of paths. Since there are $n$ translations for  any path $P_{0,d}$, $d=1,\dots,n-1$, we conclude that the load on any ring  (skip, respectively)  link in each direction is the total number of used  ring  (skip, respectively)  links in that direction by paths  $P_{0,d}$, $d=1,\dots,n-1$.

\medskip
\textsf{Claim 1:}~The maximum load on ring links under $\cal R$ is equal to the number of clockwise ring links used when $q$ is even, and anticlockwise ring links used when $q$ is odd, by paths  $P_{0,d}$, $d=1,\dots,n-1$.  

\medskip
\textsf{Claim 2:}~The maximum load under $\cal R$ on skip links is equal to the number of clockwise skip links used by paths  $P_{0,d}$, $d=1,\dots,n-1$.   

\medskip
\textit{Proof of Claims 1 and 2.}~For any path $P_{0,d}$ in ${\cal R}$, where $d < n/2 $, the path $P_{0,n - d}$ is in ${\cal R}$ and is distinct from $P_{0,d}$. Moreover, if $(u, v)$ is a link in one of these two paths, then $(n\ominus u,n\ominus v)$ is a link in the other (with opposite direction). If $n$ is odd, then $d\ne n - d$ for all $d\in\ZZ_n$, and so the number of clockwise ring links (skip links, respectively) used is equal to the number of anticlockwise ring links (skip links, respectively) used by the paths  $P_{0,d}$.

Assume $n$ is even. If $d=n/2$, then $P_{0,d}$ and $P_{0,n - d}$ are identical, and this path does not use any anticlockwise skip link. Hence $P_{0,n/2}$ uses fewer anticlockwise skip links than clockwise skip links. So the maximum load on skip links is equal to the number of clockwise skip links used by the paths $P_{0,d}$. 

If $q$ is even, then $n/2=q s /2 + r /2$; in this case $P_{0,n/2}$ uses $r/2$ clockwise ring links, and so it uses fewer anticlockwise ring links than clockwise ring links. Thus the maximum load on ring links is equal to the number of clockwise ring links used by the paths  $P_{0,d}$. 

If $q$ is odd, then $n/2 = (q-1)s /2 + (s+r)/2$; in this case $P_{0,n/2}$ uses $ (s-r)/2$ anticlockwise ring links, and so it uses fewer clockwise ring links. Hence the maximum load on ring links is equal to the number of anticlockwise ring link used by the paths  $P_{0,d}$. This completes the proof of Claims 1 and 2. 
\qed
 
We now count the number of links in $C_n(1,s)$ used by paths $P_{0,d}$, $d=1,\dots,n-1$, for even $q$ and odd $q$ separately. 

\medskip
\textsf{Ring links:}~We first count the number of clockwise (anticlockwise, respectively) ring links used by paths $P_{0,d}$ when $q$ is even (odd, respectively), where $P_{0,d}$ is defined in Construction \ref{const} in Step 1  when $d=is+j \le\nn $ and in Step 2 when $d=n-(is+j) >\nn$, where $0\le i\le \qq$ and $0\le j\le s-1$.

\medskip
\textsf{Case 1:}~$q$ is even.  If $d\le \nn$, then $is+j \le qs/2 +\rr$; if $d >\nn$, then $d=n-(is+j)$ and so $is+j<qs/2+\lceil r/2 \rceil$. 
When $d\le \nn$, by Step \ref{step:s1}, the path $P_{0,d}$ uses some clockwise ring link if and only if either $0 \le i \le q/2-1$ and $0 \le j \le \ss$, or $i=q/2$ and $0 \le j \le \rr$. When $d>\nn $, by Steps \ref{step:s1}(b) and \ref{step:s3}, $P_{0,d}$ uses some clockwise ring link if and only if $1 \le i \le q/2$ and $\ss+1 \le j \le s-1$. Moreover, $P_{0,d}$ uses $j$ ($s-j$, respectively) clockwise ring links if $d\le \nn$ ($d> \nn$, respectively). Therefore, the total number of clockwise ring links used by the paths $P_{0,d}$, $d=1,\dots,n-1$, is equal to
\be
\label{eq:even1} 
\sum_{i=0}^{q/2-1}\sum_{j=0}^{\ss} j + \sum_{i=q/2}^{q/2}\sum_{j=0}^{\rr} j  +  \sum_{i=1}^{q/2}\sum_{j=\ss+1}^{s-1} (s-j) = 
 \frac{q}{4} \left\lfloor \frac{s^2}{2} \right\rfloor + \frac{1}{2}\brrr\left\lfloor\frac{r+2}{2}\right\rfloor. 
\ee

\medskip
\textsf{Case 2:}~$q$ is odd. If $d\le \nn$, then $is +j \le (q-1)s/2 + \sr$; if $d >\nn$, then $d=n-(is+j)$ and so $is+j < (q-1)s/2 +  \lceil (s+r)/2 \rceil$. 
When $d\le \nn$, by Step \ref{step:s1}(b), $P_{0,d}$ uses some anticlockwise ring link if and only if either $1 \le i \le (q-1)/2$ and $\ss+1 \le j \le s-1$, or $i= (q+1)/2$ and $\ss+1 \le j \le \sr$. When $d > \nn$, by Steps \ref{step:s1}(a) and \ref{step:s3}, $P_{0,d}$ uses some anticlockwise ring link if and only if $0 \le i \le (q-1)/2$ and $1 \le j \le \ss$. 
The path $P_{0,d}$ uses $s-j$ ($j$, respectively) anticlockwise ring links if $d\le \nn$ ($d> \nn$, respectively). 
 Therefore, the total number of anticlockwise ring links used by the paths $P_{0,d}$, $d=1,\dots,n-1$, is given by 
\be
\label{eq:odd1}
\sum_{i=1}^{(q-1)/2}\sum_{j=\ss+1}^{s-1} (s-j) + \sum_{i=(q+1)/2}^{(q+1)/2}\sum_{j=\ss+1}^{\sr} (s-j)  +  \sum_{i=0}^{(q-1)/2}\sum_{j=1}^{\ss}j  = 
  \frac{q}{4} \left\lfloor \frac{s^2}{2} \right\rfloor  + \Delta.
\ee 

\medskip
\textsf{Skip links:}~
Now we evaluate the load on clockwise skip links. Note that $P_{0,d}$ uses clockwise skip links if $d\le \nn $. 

\medskip
\textsf{Case 3:}~$q$ is even. In this case $P_{0,d}$ uses exactly $k$  clockwise skip links if and only if  either (i) $i=k$, $0\le j \le \ss $, if $1\le k \le q/2-1$; (ii) $i=k$, $0 \le j \le \rr$, if $k=q/2$; or (iii)  $i=k-1$, $\ss +1 \le j \le s-1$, if $1\le k \le q/2$. So the total number of clockwise skip links used by the paths $P_{0,d}$, $d=1,\dots,n-1$, is equal to
\be
\label{eq:even2}
\sum_{k=1}^{q/2-1}k \left(\bsss +1\right)  + \sum_{k=q/2}^{q/2} k\left(\brrr+1\right) + \sum_{k=1}^{q/2} k\left(s - \bsss -1\right) = \\[.2cm] 
\frac{q^2s}{8} + \frac{q}{2} \left(\brrr + \frac{\epsilon(s)  }{2}\right) .
\ee

\medskip
\textsf{Case 4:}~$q$ is odd. By Step \ref{step:s1}, $P_{0,d}$ uses exactly $k$  clockwise skip links if and only if  either (i) $i=k$, $0 \le j \le \ss$, if $1\le k \le (q-1)/2$; (ii) $i=k-1$, $\ss +1\le j \le s-1$, if $1\le k \le (q-1)/2$; or (iii) $i=k-1, \ss+1 \le j \le \sr$, if $k=(q+1)/2$. Thus the total number of clockwise skip links used by the paths $P_{0,d}$, $d=1,\dots,n-1$, is equal to
\be
\label{eq:odd2}
\sum_{k=1}^{(q-1)/2} k \bssspl  + \sum_{k=1}^{(q-1)/2} k \left\lceil \frac{s-2}{2} \right\rceil + \sum_{k=(q+1)/2}^{(q+1)/2} k \left\lfloor\frac{r+\epsilon(s)}{2}\right\rfloor 
 =\frac{(q^2-1)s}{8}+ \frac{q+1}{2}\left\lfloor\frac{r+\epsilon(s)}{2}\right\rfloor.
\ee
 
Using (\ref{eq:even1})-(\ref{eq:odd2}), Claims 1 and 2 imply the required results immediately. 
\end{proof}

By comparing the two terms in each case of Lemma \ref{lem:ub}, we can identify the maximum term for different ranges of $s$, which is presented in the following lemma.

\begin{lem}
\label{lem:ub1}
The following hold:
\begin{itemize}
\item[\rm (a)] if $q$ is even and $2 \le  s \le \sqrt{n-1}$, then  
\be
\label{eq:ub}
\api(C_n(1,s),{\cal R}) \le \frac{q(n+r+2\epsilon(s))}{8};
\ee

\item[\rm (b)] if  $q$ is odd and $2 \le  s \le \sqrt{n} -1$, then 
\be
\label{eq:ub1}
\api(C_n(1,s),{\cal R}) \le \frac{q( n+ r+2\epsilon(s)) +s }{8};
\ee 
 
\item[\rm (c)] if $q=s=\sqrt{n}$, then the loads on all links of $C_n(1,s)$ are equal and
\be
\label{eq:ub2} 
\api(C_n(1,s),{\cal R}) = \frac{\sqrt{n}\ (n - \epsilon(s))}{8};
\ee
 
\item[\rm (d)] if $q$ is even and $s \ge \sqrt{n}+1$, then
\be
\label{eq:ub3} 
\api(C_n(1,s),{\cal R}) \le \frac{sn+r-\epsilon(s)q}{8};
\ee
 
\item[\rm (e)] if $q$ is odd and $s \ge \sqrt{n}$, then
\be
\label{eq:ub4}
\api(C_n(1,s),{\cal R}) \le \frac{s (n + r+ 2) - \epsilon(s)q}{8}.
\ee 
\end{itemize}
\end{lem}

Moreover, when $\sqrt{n}$ is not an integer, if $s = \lfloor\sqrt{n}\rfloor$ and $q$ is odd or $s = \lfloor\sqrt{n}+1\rfloor$ and $q$ is even, then the greater term in each case in Lemma \ref{lem:ub} relies on $r$. However, the two terms are almost equal for sufficiently large $n$.  When $\sqrt{n}$ is an odd integer and $q=s=\sqrt{n}$,  the conditions in cases (c)  and  (e) are satisfied simultaneously. Although the right hand sides of (\ref{eq:ub2}) and (\ref{eq:ub4}) are equal in this special case, the result in case (c) is slightly stronger as we have equality in (\ref{eq:ub2}).

\begin{proof}[Proof of Theorem \ref{thm:tr1}] 
As mentioned in Section \ref{subsec:two}, the lower bounds in (\ref{eq:tf1})-(\ref{eq:tf3}) follow from (\ref{eq:pi}) and Lemmas \ref{lem:lwr1}, \ref{lem:lwr2} and \ref{lem:ssqrtn} immediately. The upper bounds in (\ref{eq:tf1})-(\ref{eq:tf3}) follow from Lemma \ref{lem:ub1} and the fact that $\api(C_n(1,s)) \le \api(C_n(1,s),{\cal R})$. In fact, if $2 \le  s \le \sqrt{n} -1$, then (\ref{eq:ub}) or (\ref{eq:ub1}) applies. By (\ref{eq:ub}) and  (\ref{eq:ub1}), we obtain  $\api(C_n(1,s)) \le (q( n+ r+2\epsilon(s))/8)+(s/8) = ((n-r)(n+r+2)+s^2)/8s$.  If $s=\sqrt{n}$, then $q=\lfloor n/s \rfloor =\sqrt{n}$ and by (\ref{eq:ub2}) we obtain $\api(C_n(1,s)) \le (\sqrt{n} (n - \epsilon(s) ) / 8$.  
If $\sqrt{n}+1 \le s < n/2$, then by (\ref{eq:ub3}) and (\ref{eq:ub4}) we have $\api(C_n(1,s)) \le (s(n +r +2)-\epsilon(s)q)/8 = (s^2(n+r+2)-\epsilon(s)(n-r))/8s$.  
\end{proof}

\section{Proof of Theorem \ref{thm:tr2}}  
\label{sec:wave}

The lower bounds in (\ref{eq:thw1})-(\ref{eq:thw3}) follow from (\ref{eq:piw}) and Theorem \ref{thm:tr1} immediately. Let $\cal R$ be the routing defined in (\ref{eq:R}). Since $\aw(C_n(1,s)) \le \aw(C_n(1,s),{\cal R})$ and $w(C_n(1,s)) \le  w(C_n(1,s),{\cal R})$, it suffices to prove the upper bounds for $\aw(C_n(1,s),{\cal R})$ and $w(C_n(1,s),{\cal R})/2$.   In the following we give an arc-conflict-free colouring of ${\cal R}$ and compute the number of colours used. This number gives the required upper bound for $\aw(C_n(1,s), {\cal R})$. 
 
By Construction \ref{const}, the unique path in $\cal R$ connecting $x$ to $y=x\oplus (is+j)$ is $P_{x,y} $, which we denote by $P_{ij}^x$ in the rest of this proof, where $|i| \le \cqq$, $|j| \le \ss$ and $|is+j|\le n/2$. In other words, $P_{ij}^x$ connects $x$ to $y$ by $i$ successive skip links followed by $j$ successive ring links, where negative $i$ and $j$ stand for anticlockwise skip and ring links, respectively.  
Given $i$ and $j$, set
$$
\alpha := |j|/\gcd(s,|j|),\;\, \beta := s/\gcd(s,|j|), 
$$
where $\gcd(s,|j|)$ is the greatest common divisor of $s$ and $|j|$. Then $\alpha$ is the smallest positive integer such that $|j|$ divides $\alpha s$. Note that $\alpha s=\beta|j|$.

\begin{const}
\label{colouring}
{\em 
We define a colouring $f: {\cal R} \rightarrow \ZZ^3$ by using elements of $\ZZ^3$ as colours. 
\begin{enumerate}
\item 
If $j>0$ and $|i| \le j$, let $x_{\alpha}=\lfloor x/\alpha s\rfloor$.  
\begin{enumerate}[(a)]
\item 
If $ x \le n/2$, define $f(P^x_{ij})=(\epsilon(x_{\alpha}) j +(x+\lfloor x_{\alpha}/2\rfloor \mod{j}),i,j)$; 
\item 
if $x > n/2$, define $f(P^x_{ij})=( (2 + \epsilon(x_{\alpha})) j +(x+\lfloor x_{\alpha}/2\rfloor \mod{j}),i,j)$. 
\end{enumerate}
\item 
If $|i| > j\geq0$, let $x_0=x \mod s$.   
\begin{enumerate}[(a)]
\item 
If $x < \lfloor q/2\rfloor s$ and $x_0 \leq s - j$, define $f(P^x_{ij})=(x_0+\lfloor x/s\rfloor \mod{|i|},i,j)$;
\item if $x <(\lfloor q/2\rfloor-1)s$ and $x_0 > s - j$, define $f(P^x_{ij})=( i + (x_0 +\lfloor x/s\rfloor \mod{|i|}), i , j)$;
\item 
if $(\lfloor q/2\rfloor - 1) s \le x < \lfloor q/2\rfloor s$ and $x_0 > s - j$, define $f(P^x_{ij})=( 2i + x_0 + j - s , i , j)$;
\item 
if $\lfloor q/2\rfloor s \le x \le  n - j $ and $x_0 \leq s-j$, define $f(P^x_{ij})=(i+(x_0+\lfloor x / s\rfloor+ s-q-r -1 \mod{|i|}),i,j)$;
\item 
if $\lfloor q/2\rfloor s \le x \le  n - j $ and $x_0 > s-j$, define $f(P^x_{ij})=( x_0 +\lfloor x  / s \rfloor -q-r \mod{|i|}, i , j)$; 
\item 
if $ x >  n - j $, define $f(P^x_{ij})=( 2i + x + j - n, i , j)$. 
\end{enumerate}
\item If $j<0$, define $f(P^x_{ij})=f(P^x_{(-i)(-j)})$. 
\end{enumerate}
}
\end{const}

By the definition above, $f(P^x_{i j}) \ne f(P^{y}_{k l})$  if $k \ne i$ or $l \ne j$ when $jl \ge 0$, and if $k \ne -i$ or $l \ne -j$ when $jl < 0$. 
Since $P^x_{ij}$ and $P^y_{(-i)(-j)}$ have no common link in the same direction, to prove that $f$ is arc-conflict-free, it suffices to verify that $P^x_{ij}$ and $P^y_{ij}$ do not have any common link if $f(P^x_{i j}) = f(P^{y}_{i j})$,  or equivalently $f(P^x_{i j}) \ne f(P^{y}_{i j})$ if $P^x_{ij}$ and $P^y_{ij}$ have a common link.

Fix $i$ and $j$. Assume $x < y$. Note that $P^x_{i j}$ and $P^y_{i j}$ share a skip link if and only if $y \ominus x = h s$ for some $h$ with $0< |h| < |i|$, and they share a ring link if and only if $y \ominus x < |j|$.  

\medskip
\textsf{Case 1:}~$j>0$ and $|i|\leq j$. Assume $f(P^x_{i j})=f(P^y_{i j})$. Then by Step 1, $x_\alpha$ and $y_\alpha$ have the same parity and either $x,y \le n/2$ or $x,y> n/2$.   So if $P^x_{ij}$ and $P^y_{ij}$ share a ring link, then $y - x<j$, and if they share a skip link, then $y-x= h s$ for some $h$ with $0<h<|i|$. In the following we show that neither of these can happen, and therefore $P^x_{ij}$ and $P^y_{ij}$ cannot have any common link. 

\medskip
\textsf{Subcase 1.1:}~$x_\alpha=y_\alpha$.  We have $x+\lfloor x_\alpha/2\rfloor \equiv	 y +\lfloor y_\alpha/2\rfloor\mod{j}$ as $f(P^x_{ij})=f(P^y_{ij})$. So  $y=x+\gamma j$ as $x_\alpha=y_\alpha$, where $\gamma\ge 1$. 
Hence $P^x_{ij}$ and $P^y_{ij}$ do not share any ring link as $y - x\ge j$. If they share a skip link, then $y = x+hs$, $0<h< |i|$,  which together with $y = x+\gamma j$ gives $h s=\gamma j$. Since $x_\alpha=y_\alpha$, we have $y - x <\alpha s$ and so $h<\alpha$. The latter  inequality together with $hs=\gamma j$ contradicts the choice of $\alpha$. Hence  $P^x_{ij}$ and $P^y_{ij}$ do not share any skip link. 

\medskip
\textsf{Subcase 1.2:}~$x_\alpha\ne y_\alpha$. We have $x=x_\alpha \alpha s+l_x s+x_0$ and  $y=y_\alpha \alpha s+ l_y s+y_0$, where $0\le l_x,l_y\le \alpha-1$, $x_0= x\mod s$ and $y_0= y \mod s$. 
Since $x_\alpha$ and $y_{\alpha}$ have the same parity, we have $y_\alpha-x_\alpha\ge 2$ and so $y -  x =  (y_\alpha - x_\alpha) \alpha s + (l_y - l_x)s + y_0 - x_0 > \alpha s =\beta j \ge  j$ as $|l_y - l_x|\le  \alpha - 1$ and $|y_0 - x_0| < s$.
 So $P^x_{ij}$ and $P^y_{ij}$ do not share any ring link. 
 
Suppose by way of contradiction that $P^x_{ij}$ and $P^y_{ij}$ have a common skip link. Then $y = x  + h s$ with $h=\alpha + t$, for a positive integer $t$, as $y_\alpha - x_\alpha \ge 2$. Also $y_\alpha\alpha s+ l_y s =x_\alpha\alpha s+ l_x s +h s$ as $y_0=x_0$.   Since $f(P^x_{ij}) = f(P^y_{ij})$, we have  $ x + \lfloor  x_{\alpha}/2\rfloor \equiv y +\lfloor y_\alpha/2\rfloor \equiv x +hs+\lfloor y_\alpha/2\rfloor\mod j$. Hence $(y_\alpha - x_\alpha)/2 \equiv -h s \equiv -(\alpha s+t s) \equiv  - t c \mod j$ as $\alpha s = \beta j$, where $ c = s \mod{j}$. If $c=0$, then $y_\alpha-x_\alpha\ge 2j$ as $y_\alpha>x_\alpha$ and $(y_\alpha-x_\alpha)\alpha s \ge 2j\alpha s$. 
If $c\ne0$, then $(y_\alpha-x_\alpha)/2=k j - t c=\gcd(c,j)(k a_1 - t a_2)$, where $\gcd(a_1,a_2)=1$ and $k$ is an integer.  So $y_\alpha - x_\alpha \ge 2 \gcd(c,j)$ as $y_\alpha > x_\alpha$. Since $\alpha = j / \gcd(s,j)$ and $\gcd(s \mod j,j)=\gcd(s,j)$, we have $\alpha = j / \gcd(c,j)$, $c \ne 0$. So $ (y_\alpha - x_\alpha) \alpha s \ge 2 \gcd(c,j) \alpha s = 2 j s$. Since $(y_\alpha - x_\alpha) \alpha s = hs + (l_x - l_y) s$, we have $hs + (l_x - l_y) s \ge 2 j s$, and so $ h s \ge 2 j s + (l_y - l_x) s \ge 2 j s + (1-\alpha )s=s + (2s - \beta) j \ge (j+1) s >  |i| s$ as $\beta\le s$ and $j\ge |i|$, contradicting the fact $h < |i|$. Hence $P^x_{ij}$ and $P^y_{ij}$ have no common skip link.
 
 \medskip
\textsf{Case 2:}~$|i|>j\geq0$. We show that $P^x_{ij}$ and $P^y_{ij}$ are assigned distinct colours if they share a link. We first assume that they share a ring link, so that either $y - x \le j - 1$ or $x - y + n \le j - 1$. Denote by $f(P^x_{ij})_1$ the first coordinate of $f(P^x_{ij})$.

\medskip
\textsf{Subcase 2.1:}~$y - x \le j - 1$. Since $j\le \ss$, either $\lfloor y /s \rfloor=\lfloor x/s \rfloor+1$ or $\lfloor y /s \rfloor=\lfloor x /s \rfloor$. 
If $\lfloor y /s \rfloor=\lfloor x /s \rfloor + 1$, then $y_0 \le j$ and $x_0 > s - j$ as $y - x \le j-1$, which implies that $y$ and $x$ respectively satisfy either (a) and (b), or (d) and (c), or (d) and (e), or (f) and (e) in Step 2, and so $f(P^x_{ij})_1$ and $f(P^y_{ij})_1$ differ by at least $i$.  
Now assume $\lfloor y/s \rfloor=\lfloor x/s \rfloor$. If  $x_0 \le s-j$ and $y_0 > s-j$, then $f(P^x_{ij})_1$ and $f(P^y_{ij})_1$ differ by at least $i$; otherwise  
 $x_0 + \lfloor x/s \rfloor \not\equiv y_0 + \lfloor y/s \rfloor \mod i$ as $y_0 - x_0 \le j-1 < |i|$. So $f(P^x_{ij})\ne f(P^y_{ij})$.

\medskip
\textsf{Subcase 2.2:}~$x - y + n \le j - 1$. In this case we have $0 \le x < j$ and $n - j < y < n$, and so $f(P^x_{ij})_1$ and $f(P^y_{ij})_1$ differ by at least $2i$ by (a) and (f) in Step 2. 

Now we assume that $P^x_{ij}$ and $P^y_{ij}$ share a skip link, so that either $y=x+hs$ or $x= y +hs-n$, where $0< h <|i|$.

\medskip
\textsf{Subcase 2.3:}~$y = x+hs$. So $y_0=x_0$  and $\lfloor y/s \rfloor - \lfloor x/s \rfloor = h$. If both $x$ and $y$ satisfy the same condition in Step 2 (namely, one of (a)-(f)), then $f(P^x_{ij}) \ne f(P^y_{ij})$ since $x_0+\lfloor x /s\rfloor \ne x_0+h+\lfloor  x /s\rfloor \equiv  y_0 + \lfloor y/s \rfloor \mod |i|$. If  $\lfloor q/2 -1\rfloor s\le x  < \lfloor q/2\rfloor s$ and $y > n - j$, then $x_0=y_0 > s-j$ ($j< \ss$) and $y \ge (q-1) s$, and so $\lfloor y/s \rfloor - \lfloor x/s \rfloor \ge \cqq \ge |i| > h$, which contradicts $y = x+h s$. For other ranges of $x$ and $y$, we have $x_0=y_0$ and so $f(P^x_{ij}) \ne f(P^y_{ij})$. 

\medskip
\textsf{Subcase 2.4:}~$x = y+hs-n$. In this case we have $x < h s$ and $y> n- h s = (q-h)s +r $. Since $h< |i| \le \cqq$, we then have $0 \le x < \lceil (q - 2)/2 \rceil s$ and $\lfloor (q+2)/2\rfloor s \le y < n$.    Also we have $y_0=x_0+ r\mod{s}$ and $\lfloor y/s \rfloor = \lfloor (x+r)/s \rfloor + q - h$. 

We have the following: (i) when $r=0$ (which implies $y_0 = x_0$) or $n-j < y$, $f(P^x_{ij})_1$ and $f(P^y_{ij})_1$ differ by at least $i$ by Step 2; 
(ii) when $x_0\le s-j$ and $y_0 > s-j$, $f(P^x_{ij})_1 \equiv x_0+\lfloor x/s \rfloor\mod |i|$ and $f(P^y_{ij})_1=(x_0 + r) +(\lfloor x/s \rfloor + q - h)  -q -r \equiv x_0+\lfloor x / s \rfloor-  h \mod |i|$, and so $f(P^x_{ij})_1 \ne f(P^y_{ij})_1$;  
(iii) when $x_0> s - j$ and $y_0 \le s - j$, we have $y_0 = x_0 + r - s$ and $\lfloor y/s \rfloor = \lfloor x/s \rfloor + 1+ q - h$; hence $f(P^x_{ij})_1 \equiv i + ( x_0 + \lfloor x/s \rfloor\mod |i|)$ and $f(P^y_{ij})_1 \equiv  i+((x_0 + r  - s ) +( \lfloor x /s \rfloor+ 1+ q -  h)+s-q-r-1\mod |i|) \equiv i+(x_0+\lfloor x /s \rfloor - h \mod |i|)$, implying $f(P^x_{ij})_1 \ne f(P^y_{ij})_1$; 
 (iv) when $x_0 > s - j$ and $y_0 > s - j$, or $x_0\le  s-j$ and $y_0 \le s-j$, $f(P^x_{ij})_1$ and $f(P^y_{ij})_1$ differ by at least $i$ by Step 2. 
 
In summary, whenever $P^x_{ij}$ and $P^y_{ij}$ share a link, they are assigned different colours.
 
\medskip
\textsf{Case 3:}~$j<0$. If $f(P^x_{ij})=f(P^y_{ij})$, then $f(P^x_{(-i)(-j)})=f(P^y_{(-i)(-j)})$ by Construction \ref{colouring}. Thus, by what we proved in Cases 1 and 2, $P^x_{(-i)(-j)}$ and $P^y_{(-i)(-j)}$ do not have any common link. Therefore, $P^x_{ij}$ and $P^y_{ij}$ have no common link.   

So far we have proved that the colouring $f: {\cal R} \rightarrow \ZZ^3$ is arc-conflict-free.  

The number of colours used by $f$ is $|f({\cal R})|$. We now estimate this number and thus obtain the required upper bounds for $\aw (C_n(1,s),{\cal R})$ by using $\aw (C_n(1,s),{\cal R}) \le |f({\cal R})|$. For fixed $i$ and $j$, $f$ uses $4j$ colours if $j \ge |i|$, $2|i|+j$ colours if $0 \le j < |i|$, and no new colours if $j<0$. Note that $|i| \le \cqq$ and $|j| \le \ss$ as mentioned earlier.
Thus, if $j \ge |i|$ and $j > \cqq$, then $-\cqq \le i \le \qq$; and if $0 \le j \le |i| -1$ and $|i| -1 > \ss$, then $0 \le j \le \ss$. Therefore, by setting $\gamma_1 := \min \{\ss, \cqq\}$ and $\gamma_2 := \min\{\ss + 1, \cqq\}$ and noting $\sum_{i=-k}^k  2|i| =\sum_{i=1}^k 4 i $, we obtain 

\be
\label{eq:sumf}
|f({\cal R})| = \sum_{ j = 1}^{\gamma_1}\sum_{i = - j}^{j} 4j 
 + \sum_{ j = 1 + \gamma_1 }^{\ss} \sum_{i = - \cqq}^{\cqq} 4j  
 + \sum_{i = 1}^{\gamma_2} \sum_{j=0 }^{ i-1 } \left( 4i+2j \right) 
 +  \sum_{i =  1 +\gamma_2 }^{\cqq} \sum_{j=0}^{\ss} \left( 4 i+ 2j \right).
 \ee 
If $\ss \le \cqq - 2$, then (\ref{eq:sumf}) is equal to 
\be
\label{eq:wb1}
 \frac{1}{6}\bssspl \left( 3 q^2+ 6q + (3q+10)\bsss + 8\bsss^2\right)  + \frac{\epsilon(q)}{6}\bssspl \left( 3\bsss + q+\frac{3}{2}\right).
\ee 
If $\cqq -1 \le \ss \le \cqq$, then  (\ref{eq:sumf}) is equal to
\be
\label{eq:wb2}
 \frac{5 q^3+12 q^2 + 4q}{24} + \frac{2}{3} \bsss  \bssspl \left(4 \bsss + 5 \right) + \epsilon(q)\frac{ 5q^2 +13 q + 7}{8}.
\ee 
If $\ss \ge \cqq +1$, then  (\ref{eq:sumf}) is equal to
\be
\label{eq:wb3}
 \frac{q^3 +12 q^2 + 20 q}{24} + (2 q  + 2) \bsss \bssspl  +  \epsilon(q) \left( \frac{  q^2 + 9 q + 11 }{8} + 2 \bssspl \bsss \right).  
\ee

Recall from (\ref{eq:ka}) that $\kappa(a)=a+(\epsilon(s)+\epsilon(q))/2$. Thus, for $s = \sqrt{n-r + (\kappa(a))^2} + \kappa(a)$, we have $\ss = \cqq + a$ as $q=(n-r)/s$. Therefore, by applying $\ss \le s/2$ in (\ref{eq:wb1})-(\ref{eq:wb3}), we obtain upper bounds for $|f({\cal R})|$ which yields the upper bounds in (\ref{eq:thw1})-(\ref{eq:thw3}).

To prove that the upper bounds in (\ref{eq:thw1})-(\ref{eq:thw3}) also apply to $w(C_n(1,s),{\cal R})/2$,
we modify the definition of $f$ as follows. Define $f$ in the same as in Construction \ref{colouring} except that in Step 3 we redefine $f(P^x_{ij})=-f(P^x_{(-i)(-j)})$ for $j<0$. Obviously, $f(P^x_{i j}) \ne f(P^{y}_{k l})$  if $k \ne i$ or $l \ne j$. Moreover, when $P^x_{i j}$ and $P^{y}_{i j}$ share an edge, they share an arc and so are assigned distinct colours by the discussion above. Therefore, this modified colouring $f$ is edge-conflict-free. Since it uses twice as many colours as in the directed version, the upper bounds in (\ref{eq:thw1})-(\ref{eq:thw3}) are also upper bounds for $w(C_n(1,s),{\cal R})/2$. 
\qed

\bigskip
{\bf Acknowledgements}~~The authors appreciate Professor Graham Brightwell for his comments which led to Lemma \ref{lem:ssqrtn} and subsequent improvement of a lower bound. The authors also acknowledge an anonymous referee for her/his helpful comments. Mokhtar was supported by MIFRS and MIRS of the University of Melbourne and Zhou by a Future Fellowship (FT110100629) of the Australian Research Council. 
 
\small

\bibliographystyle{amsplain}

\begin{thebibliography}{10}

\bibitem{Amar2001}
D.~Amar, A.~Raspaud, and O.~Togni, \emph{All-to-all wavelength-routing in
  all-optical compound networks}, Discrete Mathematics \textbf{235} (2001),
  no.~1--3, 353--363.

\bibitem{Beauquier1999}
B.~Beauquier, \emph{All-to-all communication for some wavelength-routed
  all-optical networks}, Networks \textbf{33} (1999), no.~3, 179--187.

\bibitem{Beauquier1997}
B.~Beauquier, J.~C. Bermond, L.~Gargano, P.~Hell, S.~Perennes, and U.~Vaccaro,
  \emph{Graph problems arising from wavelength-routing in all-optical
  networks}, 2nd Workshop on Optics and Computer Science (WOCS), 1997.

\bibitem{Beauquier:1999b}
B.~Beauquier, S.~P{\'e}rennes, and D.~T\'{o}th, \emph{All-to-all routing and
  coloring in weighted trees of rings}, Proceedings of the 11th Annual ACM
  Symposium on Parallel Algorithms and Architectures, SPAA'99, ACM, 1999,
  pp.~185--190.

\bibitem{Bermond1995}
J.~C. Bermond, F.~Comellas, and D.~F. Hsu, \emph{Distributed loop computer
  networks: a survey}, J. Parallel and Distributed Computing \textbf{24}
  (1995), no.~1, 2--10.

\bibitem{Bermond2000}
J.~C. Bermond, L.~Gargano, S.~Perennes, A.~A. Rescigno, and U.~Vaccaro,
  \emph{Efficient collective communication in optical networks}, Theoretical
  Computer Science \textbf{233} (2000), no.~1--2, 165--189.

\bibitem{Bian2009}
Z.~Bian, Q.~Gu, and X.~Zhou, \emph{Efficient algorithms for wavelength
  assignment on trees of rings}, Discrete Applied Mathematics \textbf{157}
  (2009), no.~5, 875--889.

\bibitem{Chen2005}
B.~Chen, W.~Xiao, and B.~Parhami, \emph{Diameter formulas for a class of
  undirected double-loop networks}, J. Interconnection Networks \textbf{6}
  (2005), no.~1, 1--15.

\bibitem{Deng2003}
X.~Deng, G.~Li, W.~Zang, and Y.~Zhou, \emph{A 2-approximation algorithm for
  path coloring on a restricted class of trees of rings}, Journal of Algorithms
  \textbf{47} (2003), no.~1, 1--13.

\bibitem{Fang1998}
X.~G. Fang, C.~H. Li, and C.~E. Praeger, \emph{On orbital regular graphs and
  frobenius graphs}, Discrete Mathematics \textbf{182} (1998), no.~1--3,
  85--99.

\bibitem{Fertin2004}
G.~Fertin and A.~Raspaud, \emph{A survey on {Kn\"{o}del} graphs}, Discrete
  Applied Mathematics \textbf{137} (2004), no.~2, 173--195.

\bibitem{Gargano1997}
L.~Gargano, P.~Hell, and S.~Perennes, \emph{Colouring paths in directed
  symmetric trees with applications to {WDM} routing}, Automata, Languages and
  Programming, Lecture Notes in Computer Science, vol. 1256, 1997,
  pp.~505--515.

\bibitem{Gargano2000}
L.~Gargano and U.~Vaccaro, \emph{Routing in all-optical networks: Algorithmic
  and graph-theoretic problems}, Numbers, Information and Complexity, Springer,
  2000, pp.~555--578.

\bibitem{Gauyacq1998}
G.~Gauyacq, C.~Micheneau, and A.~Raspaud, \emph{Routing in recursive circulant
  graphs: edge forwarding index and {Hamiltonian} decomposition},
  Graph-Theoretic Concepts in Computer Science, Lecture Notes in Computer
  Science, vol. 1517, Springer, 1998, pp.~227--241.

\bibitem{Gomez2007}
D.~G\'{o}mez, J.~Gutierrez, and \'{A}. Ibeas, \emph{Optimal routing in double
  loop networks}, Theoretical Computer Science \textbf{381} (2007), no.~1--3,
  68--85.

\bibitem{Heydemann1989}
M.~C. Heydemann, J.~C. Meyer, and D.~Sotteau, \emph{On forwarding indices of
  networks}, Discrete Applied Mathematics \textbf{23} (1989), no.~2, 103--123.


\bibitem{Heydemann1997}
M.~C. Heydemann, \emph{Cayley graphs and interconnection networks}, in: G. Hahn and G. Sabidussi eds., Graph Symmetry,
Kluwer Academic Publishing, Dordrecht, 1997, pp. 167--224.

\bibitem{Hwang2001}
F.~K. Hwang, \emph{A complementary survey on double-loop networks}, Theoretical
  Computer Science \textbf{263} (2001), no.~1--2, 211--229.

\bibitem{Hwang2003}
F.~K. Hwang, \emph{A survey on multi-loop networks}, Theoretical Computer
  Science \textbf{299} (2003), no.~1--3, 107--121.

\bibitem{Kosowski2009}
A.~Kosowski, \emph{Forwarding and optical indices of a graph}, Discrete Applied
  Mathematics \textbf{157} (2009), no.~2, 321--329.

\bibitem{Mans2004}
B.~Mans and I.~Shparlinski, \emph{Bisecting and gossiping in circulant graphs},
  LATIN 2004: Theoretical Informatics, Lecture Notes in Computer Science, vol.
  2976, Springer, 2004, pp.~589--598.

\bibitem{Schroder1997}
H.~{Schr\"{o}der}, O.~S{\'y}kora, and I.~Vrt'o, \emph{Optical all-to-all
  communication for some product graphs (extended abstract)}, SOFSEM'97: Theory
  and Practice of Informatics, Lecture Notes in Computer Science, vol. 1338,
  Springer, 1997, pp.~555--562.

\bibitem{Sole1994}
P.~Sol{\'e}, \emph{The edge-forwarding index of orbital regular graphs},
  Discrete Mathematics \textbf{130} (1994), no.~1--3, 171--176.

\bibitem{Stojmenovic1997}
I.~Stojmenovi{\'c}, \emph{Multiplicative circulant networks topological
  properties and communication algorithms}, Discrete Applied Mathematics
  \textbf{77} (1997), no.~3, 281--305.

\bibitem{Thomson2008}
A.~Thomson and S.~Zhou, \emph{Frobenius circulant graphs of valency four}, J.
  Australian Mathematical Society \textbf{85} (2008), 269--282.

\bibitem{Thomson2010}
A.~Thomson and S.~Zhou, \emph{Gossiping and routing in undirected triple-loop
  networks}, Networks \textbf{55} (2010), no.~4, 341--349.

\bibitem{Thomson2012}
A.~Thomson and S.~Zhou, \emph{Frobenius circulant graphs of valency six,
  {Eisenstein-Jacobi} networks, and hexagonal meshes}, European J. of
  Combinatorics \textbf{38} (2014), 61--78.

\bibitem{Xu2013}
J.~M. Xu and M.~Xu, \emph{The forwarding indices of graphs - a survey},
  Opuscula Math \textbf{33} (2013), no.~2, 345--372.

\bibitem{Xu2007}
M.~Xu, J.-M. Xu, and L.~Sun, \emph{The forwarding index of the circulant
  networks}, J. Mathematics \textbf{27} (2007), no.~6, 623--629.

\bibitem{Yuan2007}
J.~Yuan, J.~Y. Zhang, and S.~Zhou, \emph{Routing permutations and involutions
  on optical ring networks: complexity results and solution to an open
  problem}, J. Discrete Algorithms \textbf{5} (2007), no.~3, 609--621.

\bibitem{Zerovnik1993}
J.~{\v{Z}}erovnik and T.~Pisanski, \emph{Computing the diameter in
  multiple-loop networks}, J. Algorithms \textbf{14} (1993), no.~2, 226--243.

\end{thebibliography}

\end{document}